\documentclass[12pt]{amsart} 
\usepackage{amsmath,amssymb,amscd,amsthm}
\usepackage{latexsym}
\usepackage{graphicx}
\usepackage[english]{babel}
\usepackage[latin1]{inputenc}       
\usepackage{textcomp}
\usepackage{times}
\usepackage{pgf,pgfarrows,pgfnodes,pgfautomata,pgfheaps}
\usepackage{colortbl}

\usepackage[a4paper,twoside,left=3cm,right=2.8cm,top=3.1cm,bottom=2.3cm]{geometry}

\newtheorem{theorem}{Theorem}[section]
\newtheorem{corollary}[theorem]{Corollary}
\newtheorem{proposition}[theorem]{Proposition}
\newtheorem{lemma}[theorem]{Lemma}

\newtheorem{remark}[theorem]{Remark}

\newtheorem{conjecture}[theorem]{Conjecture}

\def\irr#1{{\rm Irr}(#1)}
\def\irrr#1#2 {\irr {#1 \mid #2}}
\newcommand{\R}{\mathbb R}
\newcommand{\N}{\mathbb N}

\newcommand{\s}{\mathbb S}
\newcommand{\sfe}{{{\mathbb S}^{n-1}}}

\newcommand{\K}{\mathcal K}
\newcommand{\C}{C^{2,+}(\sfe)}
\newcommand{\trace}{{\rm tr}}

\begin{document}

\title[On the stability of Brunn-Minkowski inequalities]{On the stability
\\
of Brunn-Minkowski type inequalities}
\author[Andrea Colesanti, Galyna Livshyts, Arnaud Marsiglietti]{Andrea Colesanti, Galyna Livshyts, Arnaud Marsiglietti}
\address{Dipartimento di Matematica e Informatica ``U. Dini",
Universit\`a degli Studi di Firenze}
\email{colesant@math.unifi.it}
\address{School of Mathematics, Georgia Institute of Technology} \email{glivshyts6@math.gatech.edu}
\address{Center for the Mathematics of Information, California Institute of Technology}
\email{amarsigl@caltech.edu}
\subjclass[2010]{Primary: 52} 
\keywords{Convex bodies, log-concave, Brunn-Minkowski, Cone-measure}
\date{\today}
\begin{abstract}
We establish the stability near a Euclidean ball of two conjectured inequalities: the dimensional Brunn-Minkowski inequality for radially symmetric
log-concave measures in $\R^n$, and of the log-Brunn-Minkowski inequality. 
\end{abstract}
\maketitle

\section{Introduction}

The classical Brunn-Minkowski inequality states that for $\lambda\in[0,1]$ and for Borel measurable sets $A$ and $B$ in $\R^n$,
such that $(1-\lambda)A+\lambda B$ is measurable as well,
\begin{equation}\label{Brun-Mink}
|\lambda A+(1-\lambda)B|^{\frac{1}{n}}\geq \lambda |A|^{\frac{1}{n}}+(1-\lambda)|B|^{\frac{1}{n}}.
\end{equation}
Here $|\cdot|$ denotes the Lebesgue measure, the addition between sets is the standard vector addition, and multiplication of sets by non-negative reals
is the usual dilation.

This inequality has found many important applications in Geometry and 
Analysis  (see {\em e.g.} Gardner \cite{Gar} for an exhaustive survey on this subject). 
For example, the classical isoperimetric inequality can be deduced in a few lines
from \eqref{Brun-Mink}. Also, Maurey \cite{M} deduced from this inequality the Poincar\'e inequality for the Gaussian measure and Gaussian concentration properties. 
Based on Maurey's results, Bobkov and Ledoux proved that the Brunn-Minkowski inequality implies Brascamp-Lieb and log-Sobolev inequalities \cite{BL1}; they also 
deduced sharp Sobolev and Gagliardo-Nirenberg inequalities \cite{BL2}. A different argument was developed by the first named author in \cite{Col1} to deduce Poincar\'e
 type inequalities on the boundary of convex bodies from the Brunn-Minkowski inequality.

\medskip

Recall that a convex body is a convex compact set with non-empty interior. The family of convex bodies of $\R^n$ will be denoted by $\K^n$. For the theory of convex bodies we refer the reader to the books by Ball \cite{Ball-book}, Bonnesen, Fenchel \cite{BF}, Koldobsky \cite{Kold}, Milman and Schechtman \cite{MS}, Schneider \cite{book4} and others. A measure $\gamma$ on $\R^n$ is called log-concave if for any pair of sets $A$ and $B$ and for any scalar $\lambda\in [0,1]$,
\begin{equation}\label{log-concavity}
\gamma(\lambda A+(1-\lambda)B)\geq \gamma(A)^{\lambda}\gamma(B)^{1-\lambda}.
\end{equation}
Borell showed \cite{bor} that a measure is log-concave if it has a density (with respect to the Lebesgue measure) which is log-concave (see also Pr\'{e}kopa \cite{P}, Leindler \cite{L}). In particular, the Lebesgue measure on $\R^n$ is log-concave:
\begin{equation}\label{Brun-Mink_mult}
|\lambda A+(1-\lambda)B|\geq |A|^{\lambda}|B|^{1-\lambda}.
\end{equation}
Inequality (\ref{Brun-Mink}) implies (\ref{Brun-Mink_mult}) by the arithmetic-geometric mean inequality. Conversely, a simple argument based on the homogeneity
of Lebesgue measure shows that 
(\ref{Brun-Mink_mult}) implies (\ref{Brun-Mink}) (see, for example, \cite{Gar}). 
In general, a property analogous to (\ref{Brun-Mink}) may not hold for log-concave measures which are not homogeneous. The transposition of 
\eqref{Brun-Mink} to a measure $\gamma$,
\begin{equation}\label{Brun-Mink-gamma}
\gamma(\lambda A+(1-\lambda)B)^{\frac{1}{n}}\geq \lambda \gamma(A)^{\frac{1}{n}}+(1-\lambda)\gamma(B)^{\frac{1}{n}},\quad\forall\,\lambda\in[0,1],
\end{equation}
as $A$ and $B$ vary in some class of sets, will be called a \textbf{dimensional Brunn-Minkowski inequality}. If $\gamma$ is the Gaussian measure, $A=\{p\}$, 
$p\in\R^n$, and $B$ is measurable set with positive measure, then the set $A+B$ is the translate of $B$ by $p$. Hence, letting $|p|\to\infty$, and keeping $B$ fixed, 
(\ref{Brun-Mink-gamma}) fails. Moreover, Nayar and Tkocz \cite{polish} constructed an example in which 
(\ref{Brun-Mink-gamma}) fails for the Gaussian measure while both $A$ and $B$ contain the origin. 
Gardner and Zvavitch \cite{GZ} proved that, for the Gaussian measure, \eqref{Brun-Mink-gamma} holds if the sets $A$ and $B$ are convex symmetric dilates of each other.
They also proposed a conjecture for the Gaussian measure, that we state it in a more general form.
\begin{conjecture}[Gardner, Zvavitch -- generalized]\label{Gauss-BM}
Let $n\geq 2$. Let $\gamma$ be a log-concave symmetric measure (i.e. $\gamma(A)=\gamma(-A)$ for every measurable set $A$)
on $\R^n.$ Let $K$ and $L$ be symmetric convex bodies in $\R^n$. Then
\begin{equation}\label{Log_conc_BM}
\gamma(\lambda K+(1-\lambda)L)^{\frac{1}{n}}\geq \lambda \gamma(K)^{\frac{1}{n}}+(1-\lambda)\gamma(L)^{\frac{1}{n}}.
\end{equation}
\end{conjecture}

\medskip

Next, we pass to describe the log-Brunn-Minkowski inequality.
For a scalar $\lambda\in[0,1]$ and for convex bodies $K$ and $L$ containing the origin in their interior, with support functions $h_K$ and $h_L$, respectively
(see section \ref{preliminaries} for the definition), define their geometric average as follows:
\begin{equation}
K^{\lambda}L^{1-\lambda}:=\{x\in \R^n\,:\, \langle x,u\rangle\leq h_K^{\lambda}(u) h_L^{1-\lambda}(u)\,\,\forall u\in \sfe\},
\end{equation}
where $\langle\cdot,\cdot\rangle$ is the standard scalar product in $\R^n$. 
This set is again a convex body, whose support function is, in general,
smaller than the geometric mean of the support functions of $K$ and $L$. 
The following is widely known as log-Brunn-Minkowski conjecture (see \cite{BLYZ}).

\begin{conjecture}[B\"or\"oczky, Lutwak, Yang, Zhang]\label{Log-BM}
Let $n\geq 2$ be an integer. Let $K$ and $L$ be symmetric convex bodies in $\R^n$. Then
\begin{equation}\label{Log-BM_eq}
|K^{\lambda}L^{1-\lambda}|\geq |K|^{\lambda}|L|^{1-\lambda}.
\end{equation}
\end{conjecture}

Important applications and motivations for Conjecture \ref{Log-BM} can be found in \cite{BLYZ-1}, \cite{BLYZ-2}.

It is not difficult to see that the condition of symmetry is necessary (see \cite{BLYZ} or Remark \ref{Log-BM_shift} below). 
B\"or\"oczky, Lutwak, Yang and Zhang showed that this conjecture holds for $n=2$. Saroglou \cite{christos} and Cordero, Fradelizi, Maurey 
\cite{B-conj}  proved that (\ref{Log-BM_eq}) is true when the sets $K$ and $L$ are unconditional (i.e. they are symmetric with respect to every coordinate hyperplane). 
Rotem \cite{liran} showed that log-Brunn-Minkowski conjecture holds for complex convex bodies. Saroglou showed \cite{christos1} that the validity of Conjecture \ref{Log-BM} 
would imply the same statement for every log-concave symmetric measure $\gamma$ on $\R^n$: for every symmetric $K,L\in\K^n$ and for every $\lambda\in[0,1]$,
\begin{equation}\label{Log-BM-general}
\gamma(K^{\lambda}L^{1-\lambda})\geq \gamma(K)^{\lambda}\gamma(L)^{1-\lambda}.
\end{equation}

Note that the straightforward inclusion
$$
K^{\lambda}L^{1-\lambda}\subset \lambda K+(1-\lambda)L
$$ 
tells us that (\ref{Log-BM-general}) is stronger than (\ref{log-concavity}), for every measure.

In \cite{LMNZ} the second and third named authors, Nayar and Zvavitch showed that (\ref{Log-BM-general}) implies (\ref{Log_conc_BM}) for every ray-decreasing 
measure $\gamma$ on $\R^n$ and for every pair of convex sets $K$ and $L$. Therefore, Conjecture \ref{Gauss-BM} holds on the plane and for unconditional sets.

The main results of this paper are the two theorems below.

\begin{theorem}[The dimensional Brunn-Minkowski inequality near a ball]\label{BM_stability}
Let $\gamma$ be a rotation invariant log-concave measure on $\R^n$.
Let $R\in (0,\infty)$. Let $\psi\in C^2(\sfe)$. Then there exists a sufficiently small $a>0$ such that for every $\epsilon_1,\epsilon_2\in (0,a)$ and for every $\lambda\in [0,1],$ one has
$$\gamma(\lambda K_1+(1-\lambda)K_2)^{\frac{1}{n}}\geq \lambda\gamma(K_1)^{\frac{1}{n}}+(1-\lambda)\gamma(K_2)^{\frac{1}{n}},$$
where $K_1$ is the convex set with the support function $h_1=R+\epsilon_1 \psi$ and $K_2$ is the convex set with the support function $h_2=R+\epsilon_2 \psi$.
\end{theorem}

\begin{theorem}[The log-Brunn-Minkowski inequality near a ball]\label{Log-BM_stability}
Let $\gamma$ be a rotation invariant log-concave measure on $\R^n$.
Let $R\in (0,\infty)$. Let $\varphi\in C^2(\sfe)$ be \textbf{even} and strictly positive. Then there exists a sufficiently small $a>0$ such that for every $\epsilon_1,\epsilon_2\in (0,a)$ and for every $\lambda\in [0,1],$ one has
$$\gamma(K_1^{\lambda} K_2^{1-\lambda})\geq \gamma(K_1)^{\lambda}\gamma(K_2)^{1-\lambda},$$
where $K_1$ is the convex set with the support function $h_1=R\varphi^{\epsilon_1}$ and $K_2$ is the convex set with the support function $h_2=R\varphi^{\epsilon_2}$.
\end{theorem} 
Theorem \ref{Log-BM_stability} can be used to obtain a local uniqueness result for log-Minkowski problem (see B\"or\"oczky, Lutwak, Yang, Zhang \cite{BLYZ}, \cite{BLYZ-1} and the references therein), and the corresponding investigation shall be carried out in a separate manuscript.

\begin{remark}\label{Log-BM_shift}
Theorems \ref{BM_stability} and \ref{Log-BM_stability} indicate a difference between the local behaviors of the dimensional Brunn-Minkowski inequality
and the log-Brunn-Minkowski inequality. Indeed, \eqref{Log-BM_eq} fails for the simplest possible odd perturbation: the shift (which is 
equivalent to chosing $\varphi$ as the restriction of a linear function to $\sfe$). In contrast, by Theorem \ref{BM_stability} the Brunn-Minkowski inequality holds for
radially symmetric log-concave measures when $K$ and $L$ are perturbations, non necessarily even, of $RB_2^n$. 
\end{remark}

\medskip

This paper is structured as follows. Section \ref{preliminaries}
contains some preliminary material for the subsequent part of the paper. In Section \ref{Equivalence lemmas} we discuss the relations between 
the dimensional Brunn-Minkowski inequality and the log-Brunn-Minkowski inequality and their infinitesimal forms. Theorems \ref{BM_stability} and \ref{Log-BM_stability} are proved in Sections \ref{last section} and \ref{very last section}, respectively. Finally, we provide some technical details in the Section \ref{appendix}.

\subsection{Acknowledgement.} The second author would like to thank Fedor Nazarov and Artem Zvavitch for useful discussions. The second author would also like to thank the University of Florence, Italy for the hospitality. The third author would like to thank Georgia Institute of Technology for the hospitality, and was supported in part by the Walter S. Baer and Jeri Weiss CMI Postdoctoral Fellowship. The authors are thankful to the anonymous reviewer for valuable suggestions which helped to improve the presentation of this paper.

\section{Preliminaries}\label{preliminaries}

We work in the $n-$dimensional Euclidean space $\R^n$ with norm $|\cdot|$ and scalar product $\langle \cdot, \cdot\rangle$. We set 
$B_2^n:=\{x\in\R^n\,:\, |x|\leq 1\}$ and $\sfe:=\{x\in\R^n\,:\, |x|=1\}$, to denote the unit ball and the unit sphere, respectively. We shall denote the Lebesgue measure (the {\em volume}) in $\R^n$ by $|\cdot|$. 

We say that a set $A\subset \R^n$ is symmetric if for every $x\in A$ one has $-x\in A$. All measures under consideration will be tacitly assumed to be Radon measures, and all sets will be assumed to be measurable. A measure $\gamma$ on $\R^n$ is called symmetric if for every set 
$S\subset \R^n,$ $\gamma(S)=\gamma(-S)$. If the measure has a density then it is symmetric whenever the density is an even function.

A measure $\gamma$ on $\R^n$ is said to be rotation invariant if for every set $A\subset \R^n$, and for every rotation $T$, 
$\gamma(A)=\gamma(TA)$. If a rotation invariant measure $\gamma$ has a density $F$, we may write $F$ in the form:
$$
F(x)=f(|x|),
$$
for a suitable $f\,:\,[0,\infty)\to[0,\infty)$.

For $K\in\K^n$, the {\em support function} of $K$, $h_K: \sfe\rightarrow \R$, is defined as
$$
h_K(u)=\sup_{x\in K} \langle x, u\rangle.
$$
By the geometric viewpoint, $h_K(u)$ represents the (signed) distance from the origin of the supporting hyperplane to $K$ with outer unit normal $u$. 
We shall use the notation $H_K(x)$ for the 1-homogenous extension of $h_K$, that is,
$$
H_K(x)=
\left\{
\begin{array}{lll}
|x|\,h_K\left(\dfrac{x}{|x|}\right)\quad&\mbox{if $x\ne0$,}\\
0\quad&\mbox{if $x=0$.}
\end{array}
\right.
$$
The function $H_K$ is convex in $\R^n$, for every $K\in\K^n$. Vice versa, for every continuous 1-homogeneous convex function $H$ on $\R^n$, there exists a unique convex body $K$ such that $H=H_K$.

Note that $K\in\K^n$ contains the origin (resp., in its interior) if and only if $h_K \ge 0$ (resp. $h_K>0$) on $\sfe$. 
For convex bodies $K$ and $L$, and for $\alpha$, $\beta\ge0$, we have:
\begin{equation}\label{Mink_support}
h_{\alpha K+\beta L}(u)=\alpha h_K(u)+\beta h_L(u).
\end{equation}

We say that a convex body $K$ is $C^{2,+}$ if $\partial K$ is of class $C^2$ and the Gauss curvature is strictly positive at every $x\in\partial K$. 
In particular, if $K$ is $C^{2,+}$ then it admits outer unit normal $\nu_K(x)$ at every boundary point $x$. Recall that the Gauss map $\nu_K\,:\,\partial K\to\sfe$ is the map assigning the unit normal to each point of $\partial K.$ 

$C^{2,+}$ convex bodies can be characterized through their support function.
We recall that an orthonormal frame on the sphere is a map which associates a collection of $n-1$ orthonormal vectors to every point of $\sfe$. Let $\psi\in C^2(\sfe)$. We denote by $\psi_i(u)$ and $\psi_{ij}(u)$, $i,j\in\{1,\dots,n-1\}$, 
the first and second covariant derivatives of $\psi$ at $u\in\sfe$, with respect to a fixed local orthonormal frame on an open subset of $\sfe$. We define the matrix
\begin{equation}\label{curvature_matrix}
Q(\psi;u)=(q_{ij})_{i,j=1,\dots,n-1}=\left(
\psi_{ij}(u)+\psi(u)\delta_{ij}
\right)_{i,j=1,\dots,n-1},
\end{equation}
where the $\delta_{ij}$'s are the usual Kronecker symbols. On an occasion, instead of $Q(\psi;u)$ we write $Q(\psi)$. Note that $Q(\psi;u)$ is symmetric by standard properties of covariant derivatives. The meaning of this matrix becomes particularly important when 
$\psi$ is the support function of a convex body $K$. In this case we shall call it {\em curvature matrix} of $K$ (see the following Remark \ref{Jacobian}). 
The proof of the following proposition can be deduced from Schneider \cite[Section 2.5]{book4}.

\begin{proposition}\label{equivalent condition} 
Let $K\in\K^n$ and let $h$ be its support function. Then $K$ is of class $C^{2,+}$ if and only if $h\in C^2(\sfe)$ and 
$$
Q(h;u)>0\quad\forall\, u\in\sfe.
$$
\end{proposition}
In view of the previous results it is convenient to introduce the following set of functions
$$
\C=\{h\in C^2(\sfe)\,:\,Q(h;u)>0\,\forall\, u\in\sfe\}.
$$
Hence $\C$ is the set of support functions of convex bodies of class $C^{2,+}$. 

\begin{remark}\label{Jacobian}
Let $K$ be a $C^{2,+}$ convex body. Then $\nu_K\,:\partial K\to\sfe$ is a diffeomorphism. 
The matrix $Q(h;u)$ represents the inverse of the Weingarten map at $x=\nu_K^{-1}(u)$, and its eigenvalues are the principal radii of curvature of $\partial K$ at $x$. Consequently we have
$$
\det(Q(h;u))=\frac1{G(x)}
$$
where $G$ denotes the Gauss curvature.
\end{remark}

Let $K$ be a $C^{2,+}$ convex body, with support function $h_K$ and its homogenous extension $H_K$.
$H_K$ is of class $C^1(\R^n\setminus\{0\})$. By $\nabla H_K$ we denote its gradient with respect to Cartesian coordinates. The following useful relation holds: for every $u\in\sfe$, $\nabla H_K(u)$ is the (unique) point on $\partial K$ where the outer unit normal is $u$:
$$
\nabla H_K(u)=\nu_K^{-1}(u)\quad\forall\,u\in\sfe.
$$ 
In other words,
$$\langle \nabla H_K(u), \nu_K(u)\rangle=H_K(u)\quad\forall\,u\in\sfe.$$

\begin{remark}\label{gradients} Let $\psi\in C^1(\sfe)$. The notation $\nabla_\sigma \psi$ stands for the {\em spherical gradient} of $\psi$, i.e. the vector
$(\psi_1,\dots,\psi_{n-1})$, where $\psi_i$ are the covariant derivatives of $\psi$ with respect to the $i$-th element of a fixed orthonormal system
on $\sfe$. Let $\Phi$ be the 1-homogeneous extension of $\psi$ to $\R^n$. Then we have
\begin{equation}
|\nabla\Phi(u)|^2=\psi^2(u)+|\nabla_\sigma\psi(u)|^2
\end{equation}
for every $u\in\sfe$.
\end{remark}

\section{Infinitesimal versions of inequalities.}\label{Equivalence lemmas}

We denote the family of centrally symmetric convex bodies by $\K^n_s$. The notation $C^{2,+}_e(\sfe)$ will stand for the set of support functions of centrally symmetric $C^{2,+}$ convex bodies, i.e. functions from $C^{2,+}(\sfe)$ which are additionally even.

Let $h$ be the support function of a $C^{2,+}$ convex body $K$, and let $\psi\in C^2(\sfe)$; then, by Proposition
\ref{equivalent condition},
\begin{equation}\label{additive perturbation}
h_s:=h+s\psi\in C^{2,+}(\sfe)
\end{equation}
if $s$ is sufficiently small, say $|s|\leq a$ for some appropriate $a>0$. Hence for every $s$ in this range there 
exists a unique $C^{2,+}$ convex body $K_s$ with the support function $h_s$. 
For an interval $I$, we define the one-parameter family of convex bodies: 
$$
\textbf{K}(h,\psi,I):=\{K_s\,:\, h_{K_s}=h+s\psi,\, s\in I\}.
$$

\begin{lemma}\label{key_lemma_BM}
Assume that $\gamma$ is a symmetric log-concave measure with continuously differentiable density. Conjecture \ref{Gauss-BM} holds for $\gamma$ if and only if for every one-parameter family $\textbf{K}(h,\psi,I)$, with even $h$ and $\psi$, 
\begin{equation}\label{dimbmeq}
\left. \frac{d^2}{ds^2}\left[\gamma(K_s)\right]\right|_{s=0}\cdot \gamma(K_0)\le \frac{n-1}{n}\left(\left. \frac{d}{ds}\left[\gamma(K_s)\right]\right|_{s=0}\right)^2.
\end{equation}
In particular, if (\ref{dimbmeq}) holds for $K_s$ in a fixed family $\textbf{K}(h,\psi,I)$, then Conjecture \ref{Gauss-BM} holds for all sets $K_s$ in that family.
\end{lemma}
\begin{proof}
Assume first that $\gamma$ satisfies \eqref{Log_conc_BM} on the system $\textbf{K}(h,\psi,I)$. 
Then the equality $h_{K_s}=h+s\psi$, $s\in I$, and the linearity of support function with respect to Minkowski addition, 
imply that for every $s,t\in I$ and for every $\lambda\in[0,1]$
$$
K_{\lambda s+(1-\lambda)t}=\lambda K_s+(1-\lambda)K_t.
$$
By (\ref{Log_conc_BM}),
$$
\gamma(K_{\lambda s+(1-\lambda)t})^{\frac{1}{n}}=\gamma(\lambda K_s+(1-\lambda)K_t)^{\frac{1}{n}}\geq 
\lambda \gamma(K_s)^{\frac{1}{n}}+(1-\lambda)\gamma(K_t)^{\frac{1}{n}},
$$
which means that the function $\gamma(K_s)^{\frac{1}{n}}$ is concave on $I$. Inequality \eqref{dimbmeq} follows.

Conversely, suppose that for every system $\textbf{K}(h,\psi,I)$ the function $\gamma(K_s)^{\frac{1}{n}}$ has 
non-positive second derivative at $0$, i.e. \eqref{dimbmeq} holds. We observe that this implies concavity of $\gamma(K_s)^{\frac{1}{n}}$ on 
the entire interval $I$. Indeed, given $s_0$ in the interior of $I$, 
consider $\tilde{h}=h+s_0\psi$, and define a new system $\mathbf{\tilde{K}}(\tilde{h},\psi,J)$, where $J$ is a new interval such that
$\tilde{h}+s\psi=h+(s+s_0)\psi\in C^{2,+}$ for every $s\in J$. Then the second derivative of
$\gamma(K_s)^{\frac{1}{n}}$ at $s=s_0$ is negative, as it is equal to the second derivative of $\gamma(\tilde{K}_s)^{\frac{1}{n}}$ at $s=0$.
On the other hand, the concavity $\gamma(K_s)^{\frac{1}{n}}$ on the family $\textbf{K}(h,\psi,I)$ is equivalent to the validity of \eqref{Log_conc_BM} on this 
family.
\end{proof}

\medskip

A similar approach can be used for the log-Brunn-Minkowski inequality. In order to do this we introduce a corresponding type of one-parameter families 
of convex bodies. In this case, additive perturbations are replaced by multiplicative perturbations.

Let $h\in C^{2,+}(\sfe)$ and $\varphi\in C^2(\sfe)$, with $\varphi>0$ on $\sfe$. Then there exists $a>0$ such that
$$
h_s:=h\,\varphi^s\in C^{2,+}(\sfe)\quad\forall\, s\in[-a,a].
$$
In particular, by Proposition \ref{equivalent condition}, for every $s\in[-a,a]$ there exists a $C^{2,+}$ convex body $Q_s$ whose support function is $h_s$.

We introduce the corresponding 1-dimensional systems. 
$$
\textbf{Q}(h,\varphi,I):=\{Q_s\in\K^n\,:\, h_{Q_s}=h\varphi^s,\,s\in I\}.
$$
\begin{lemma}\label{key_lemma_Log-BM}
Let $\gamma$ be a symmetric log-concave measure with continuously differentiable density.  Assume that Conjecture \ref{Log-BM} holds for a measure $\gamma$, i.e. \eqref{Log-BM-general} is valid for every pair of symmetric convex sets $K$ and $L$ and for every $\lambda\in[0,1]$.
Then for every one-parameter family $Q_s\in \textbf{Q}(h,\varphi,I)$, with $h$ and $\varphi$ even, 
\begin{equation}\label{concavity at zero2}
\left.\frac{d^2}{ds^2}\log(\gamma(Q_s))\right|_{s=0}\le0.
\end{equation}
The converse is true locally: if (\ref{concavity at zero2}) holds for all $Q_s$ in a fixed family $\textbf{Q}(h,\varphi,I)$, then Conjecture \ref{Log-BM} holds for all sets $Q_s$ in $\textbf{Q}(h,\varphi,[0,\epsilon])$ for a small enough interval $[0,\epsilon]\subset I$.
\end{lemma}
\begin{proof}
Let $h\in C^{2,+}(\sfe)$ and $\varphi\in C^{2}(\sfe)$ be strictly
positive even functions on $\sfe$; there exists $a>0$ such that $h_s:=h \varphi^s$ is the support function of a convex body $Q_s$
for all $s\in [-a,a]$. Note that for $s,t\in [-a,a]$ we get 
$$
h_{\lambda s+(1-\lambda) t} = h_s^{\lambda} h_{t}^{1-\lambda},
$$
and thus
$$
Q_{\lambda s+(1-\lambda) t} = Q_s^{\lambda} Q_{t}^{1-\lambda}.
$$
If the Conjecture \ref{Log-BM} is true, then
$$
\gamma(Q_{\lambda s+(1-\lambda) t}) = \gamma(Q_s^{\lambda} Q_{t}^{1-\lambda})\geq \gamma(Q_s)^{\lambda} \gamma(Q_{t})^{1-\lambda},
$$
which means that $\gamma(Q_s)$ is log-concave in $[-a,a]$.
\end{proof}

\section{Proof of Theorem \ref{BM_stability}}\label{last section}
The following Lemma is the key step in proving Theorem \ref{BM_stability}. To prove it, we express a measure of a convex set in terms of its support function and run a long and technical computation, involving integration by parts; the complete proof is outlined in the Section \ref{appendix}.
\begin{lemma}\label{L-T-1}
Let $R>0$. Let $\gamma$ be a rotation invariant measure with density $f(|x|)$, and let $A=\int_0^1t^{n-1}f(Rt)dt.$ In the case $h_K=R$, (\ref{dimbmeq}) is equivalent to the validity of the following inequality for every $\psi\in C^{2}(\sfe)$:
\begin{equation}\label{BM_main_ineq}
\begin{aligned}
& \frac{Af(R)}{|\sfe|} \left( (n-1) \int_\sfe\psi^2du - \int_\sfe|\nabla_\sigma\psi|^2du \right) + \frac{AR f'(R)}{|\sfe|}\int_\sfe\psi^2du \leq \\
&\frac{n-1}{n}f(R)^2 \left(\frac{1}{|\sfe|}\int_\sfe\psi du\right)^2.
\end{aligned}
\end{equation}
\end{lemma}
By Lemma \ref{key_lemma_BM}, to prove the Theorem, it suffices to show the validity of (\ref{BM_main_ineq}). Let us denote the quadratic operators appearing in the left-hand side and in the right-hand side of the inequality (\ref{BM_main_ineq}) by $B_1(\psi)$ and $B_2(\psi)$, correspondingly. That is,
$$B_1(\psi)=\frac{Af(R)}{|\sfe|} \left( (n-1) \int_\sfe\psi^2du - \int_\sfe|\nabla_\sigma\psi|^2du \right) + \frac{AR f'(R)}{|\sfe|}\int_\sfe\psi^2du,$$
and
$$B_2(\psi)=\frac{n-1}{n}f(R)^2 \left(\frac{1}{|\sfe|}\int_\sfe\psi du\right)^2.$$
The next step is to decompose $\psi$ as the sum of a constant function and a function which is orthogonal to constant functions. Let us write
$$
\psi=\psi_0+\psi_1
$$
where
$$
\psi_0=\frac1{|\sfe|}\int_\sfe\psi du \quad \mbox{and} \quad \int_\sfe\psi_1 du=0.
$$
Note that
$$
\int_\sfe\psi^2d\sigma=\int_\sfe \psi_0^2d\sigma+\int_\sfe\psi_1^2d\sigma.
$$
Therefore,
$$
B_1(\psi)=B_1(\psi_0)+B_1(\psi_1),
$$
as well as
$$
B_2(\psi)=B_2(\psi_0)+B_2(\psi_1).
$$

Since $\gamma$ is radially symmetric, one has $f' \leq 0$. Moreover, by the standard Poincar\'e inequality on the unit sphere,
\begin{equation}
(n-1)\int_\sfe\psi^2du - \int_\sfe|\nabla_\sigma\psi|^2du \leq 0,
\end{equation}\label{Poincare}
for every $\psi$ such that
\begin{equation}\label{condition}
\int_\sfe\psi du = 0.
\end{equation}
Thus
$$
B_1(\psi_1)\le0=B_2(\psi_1).
$$
To prove \eqref{BM_main_ineq} it remains to show that
\begin{equation}\label{main4}
B_1(\psi_0)\le B_2(\psi_0).
\end{equation}
This condition is equivalent to
\begin{equation}\label{BM_balls}
\gamma(\lambda r_1 B_2^n+(1-\lambda) r_2 B_2^n)^{\frac{1}{n}}\geq \lambda \gamma( r_1 B_2^n)^{\frac{1}{n}}+(1-\lambda)\gamma(r_2 B_2^n)^{\frac{1}{n}},
\end{equation}
for some $r_1, r_2\in [R, R+\epsilon]$. As was shown in \cite{LMNZ} (see also the third named author \cite{arno}), this statement follows from log-Brunn-Minkowski conjecture in the case of log-concave spherically invariant measures and when $K$ and $L$ are Euclidean balls. The latter is indeed true: it follows from the results of \cite{B-conj} and \cite{christos}. 

\section{Proof of the Theorem \ref{Log-BM_stability}}\label{very last section}

As before, we start with a Lemma, which shall be rigorously proved in Section \ref{appendix}. 

\begin{lemma}\label{L-T-2}
Let $R>0$. Let $\gamma$ be a rotation invariant measure with density $f(|x|)$, and let $A=\int_0^1t^{n-1}f(Rt)dt.$ In the case $h_K=R$, (\ref{concavity at zero2}) is equivalent to the following inequality:
\begin{equation}\label{Log-BM_main_ineq}
\begin{aligned}
&A\left[nf(R)+R f'(R)\right]\frac{1}{|\sfe|}\int_\sfe\psi^2du-Af(R)\frac{1}{|\sfe|}\int_\sfe|\nabla_\sigma\psi|^2du\leq \\
&f(R)^2 \left(\frac{1}{|\sfe|}\int_\sfe\psi d\sigma\right)^2,
\end{aligned}
\end{equation}
for every even $\psi\in C^2(\sfe)$.
\end{lemma}

We follow the argument of the previous section and split the proof into two cases.

\textbf{Case 1.} Consider an even $\psi\in C^{2}(\sfe)$ such that $\int \psi=0$. Here we use some basic facts from the theory of spherical harmonics,
which can be found, for instance in \cite[Appendix]{book4}, where the reader will find hints to the corresponding literature. We denote by $\Delta_\sigma$
the spherical Laplace operator (or Laplace-Beltrami operator), on $\sfe$. 
The first eigenvalue of $\Delta_\sigma$ is $0$, and the corresponding eigenspace if formed by constant functions. Hence the 
zero-mean condition on $\psi$ implies that $\psi$ is orthogonal to such eigenspace. The second eigenvalue of $\Delta_\sigma$ is $n-1$, and the corresponding
eigenspace is formed by the restrictions of linear functions of $\R^n$ to $\sfe$. As each of them is odd and $\psi$ is even, $\psi$ is orthogonal to this eigenspace as well.
Finally, the third eigenvalue is $2n$.  Then the inequality (\ref{Log-BM_main_ineq}) amounts to
\begin{equation}\label{main_logbm}
\frac{1}{|\sfe|}\int_\sfe\psi^2du\leq \frac{f(R)}{nf(R)+R f'(R)}\frac{1}{|\sfe|}\int_\sfe|\nabla_\sigma\psi|^2du.
\end{equation}
Hence
\begin{equation}\label{Poincare_2n}
\frac{1}{|\sfe|}\int_\sfe\psi^2du\leq \frac{1}{2n}\frac{1}{|\sfe|}\int_\sfe|\nabla_\sigma\psi|^2du.
\end{equation}
Since $f$ is decreasing, we have $f'(R)\leq 0$, and hence
\begin{equation}\label{estimate_for_logbm}
\frac{f(R)}{nf(R)+R f'(R)}\geq \frac{1}{n}>\frac{1}{2n}.
\end{equation}
The inequalities (\ref{Poincare_2n}) and (\ref{estimate_for_logbm}) imply (\ref{main_logbm}).

\textbf{Case 2.} Let $\psi$ be a constant function.The inequality (\ref{Log-BM_main_ineq}) holds for constant functions because, once again, the log-Brunn-Minkowski inequality holds in the case of spherically invariant measures and Euclidean balls. 

To summarize, we established (\ref{Log-BM_main_ineq}) separately for constant functions and centered functions. A polarization argument analogous to the one presented in the proof of Theorem \ref{BM_stability} finishes the proof.

\section{Auxiliary results}\label{appendix}

\subsection{A formula expressing a measure of a convex set in terms of its support function}\label{Volume formula}

Let $\gamma$ be a probability measure on $\R^n$; we assume that $\gamma$ has a density $F$ with respect to the Lebesgue measure, and that 
$F$ is sufficiently regular ({\em e.g.} continuous). We leave the proof of the lemma below to the reader, as it is a standard argument involving polar coordinates.

\begin{lemma}\label{formula} Let $K$ be a $C^{2,+}$ convex body; let $h$ and $H$ be the support  
function of $K$ and its homogenous extension, respectively. Assume that the 
origin is in the interior of $K$. Then
\begin{equation}\label{volume formula2}
\gamma(K)=\int_{\s^{n-1}} h(y)\det Q(h;y)\int_0^{1} t^{n-1} F\left(t\nabla H(y)\right)dt dy.
\end{equation}
\end{lemma}


\subsection{The cofactor matrix and related notions}

Let $M=(m_{ij})$ be an $N\times N$ symmetric matrix, $N\in\N$. We define $C[M]$, the {\em cofactor matrix} of $M$, as follows
$$
C[M]=(c_{ij}[M])_{i,j=1,\dots,N}\quad\mbox{where}
\quad
c_{ij}[M]=\frac{\partial\det}{\partial m_{ij}}(M)\quad
i,j=1,\dots,N.
$$ 
$C[M]$ is an $N\times N$ symmetric matrix.  Using the homogeneity of the determinant we get 
\begin{equation}\label{homog1}
\sum_{i,j=1}^N c_{ij}[M]m_{ij}=N\,\det(M).
\end{equation}
 
We shall also consider the second derivatives of the determinant of a matrix with respect to its entries:
$$
c_{ij,kl}[M]=\frac{\partial^2\det}{\partial m_{ij}\partial m_{kl}}(M).
$$
By homogeneity we have that, for every $i,j=1,\dots,N$
\begin{equation}\label{homog2}
\sum_{k,l=1}^N c_{ij,kl}[M]\,m_{kl}=(N-1)c_{ij}[M].
\end{equation}

\subsection{The Cheng-Yau lemma and an extension}

Let $h\in\C$, and assume additionally that $h\in C^3(\sfe)$. Consider the cofactor matrix $y\to C[Q(h;y)]$. This is a matrix of functions on $\sfe$. The lemma of Cheng and Yau asserts
that each column of this matrix is divergence-free.

\begin{lemma}[Cheng-Yau.] Let $h\in\C\cap C^3(\sfe)$. Then, for every index $j\in\{1,\dots,n-1\}$ and for every $y\in\sfe$,
$$
\sum_{i=1}^{n-1}\left(
c_{ij}[Q(h;y)]
\right)_i=0,
$$
where the sub-script $i$ denotes the derivative with respect to the $i$-th element of an orthonormal frame on $\sfe$. 
\end{lemma}

For simplicity of notation we shall often write $C(h)$, $c_{ij}(h)$ and $c_{ij,kl}(h)$ in place of $C[Q(h)]$, $c_{ij}[Q(h)]$ and $c_{ij,kl}[Q(h)]$ respectively. 

As a corollary of the previous result we have the following integration by parts formula. 
If $h\in\C\cap C^3(\sfe)$ and $\psi,\phi\in C^2(\sfe)$, then
\begin{equation}\label{ibp1}
\int_\sfe \phi\,c_{ij}(h)(\psi_{ij}+\psi\,\delta_{ij})dy
=\int_\sfe \psi\,c_{ij}(h)(\phi_{ij}+\phi\,\delta_{ij})dy.
\end{equation}

The Lemma of Cheng and Yau admits the following extension (see the paper by the first-named author, Hug and Saorin-Gomez \cite{Colesanti-Hug-Saorin2}).
\begin{lemma} Let $\psi\in C^2(\sfe)$ and $h\in\C\cap C^3(\sfe)$. Then, for every $k\in\{1,\dots,n-1\}$ and for every $y\in\sfe$
$$
\sum_{i=1}^{n-1}\left(
c_{ij,kl}[Q(h;y)](\psi_{ij}+\psi\delta_{ij})
\right)_l=0.
$$
\end{lemma}

Correspondingly we have, for every $h\in\C\cap C^3(\sfe)$, $\psi,\varphi,\phi\in C^2(\sfe)$ and $i,j\in\{1,\dots,n-1\}$
\begin{eqnarray}\label{ibp2}
&&\int_\sfe \psi\,c_{ij,kl}(h)(\varphi_{ij}+\varphi\delta_{ij})((\phi)_{kl}+\phi\,\delta_{kl})dy\nonumber\\
&&=\int_\sfe \phi\,c_{ij,kl}(h)(\varphi_{ij}+\varphi\delta_{ij})((\psi)_{kl}+\psi\,\delta_{kl})dy.
\end{eqnarray}

\subsection{Proof of the Lemma \ref{L-T-1}}

As usual, $\gamma$ is a radially symmetric log-concave measure on $\R^n$, with density $F$ with respect to Lebesgue measure; in particular, we write $F$ in the form:
$$ 
F(x) = f(|x|). 
$$
We will assume that $f$ is smooth, more precisely $f \in C^2([0, \infty))$. Let us fix $h\in\C$ and let $K$ be a convex body with support function $h$. 
Let $\psi\in C^2(\sfe)$ and consider the one-parameter system of convex bodies ${\bf K}(h,\psi,[-a,a])$ for a suitably small $a>0$. 
In particular for every $s\in[-a,a]$ there exists a convex body $K_s$ such that $h_{K_s}=h_s$. Hence we may consider the function
$$
g\,:\,[-a,a]\to\R,\quad
g(s)=\gamma(K_s).
$$
The aim of this subsection is to derive formulas for the first and second derivative of $g(s)$ at $s=0$.
We start from the expression:
$$
g(s)=\int_\sfe h_s(u) \,\det(Q(h_s;u))\int_0^1 t^{n-1}f(t\sqrt{h_s^2(u)+|\nabla_\sigma h_s(u)|^2})dtdu,
$$
where we used Lemma \ref{formula}, the rotation invariance of $\gamma$, and Remark \ref{gradients}. To simplify notations we set
\begin{eqnarray*}
&&Q_s=Q(h_s;u)\,,\quad Q=Q_0;\quad
D_s=\left[h_s^2(u)+|\nabla_\sigma h_s(u)|^2\right]^{1/2},\quad
D=D_0;\\
&&A_s=\int_0^1t^{n-1}f(tD_s)dt\,,\quad A=A_0;\quad
B_s=\int_0^1t^{n}f'(tD_s)dt\,,\quad B=B_0;\\
&&C_s=\int_0^1t^{n+1}f''(tD_s)dt\,,\quad C=C_0.
\end{eqnarray*}
Then
\begin{eqnarray}\label{first derivative}
g'(s)&=&\int_\sfe\psi\det(Q_s)A_sdu+
\int_\sfe h_sc_{ij}(h_s)(\psi_{ij}+\psi\delta_{ij})A_sdu\nonumber\\
&+&\int_\sfe h_s\det(Q_s) B_s\frac{h_s\psi+\langle\nabla_\sigma h_s,\nabla_\sigma\psi\rangle}{D_s}du.
\end{eqnarray}
Passing to the second derivative (for $s=0$) we get
\begin{eqnarray}
g''(0)&=&2\int_\sfe\psi c_{ij}(h)(\psi_{ij}+\psi\delta_{ij})Adu\nonumber\\ &+& 2\int_\sfe\psi\det(Q)B
\frac{h\psi+\langle\nabla_\sigma h,\nabla_\sigma\psi\rangle}{D}du\nonumber\\
&+&2\int_\sfe h c_{ij}(h)(\psi_{ij}+\psi\delta_{ij})B\frac{h\psi+\langle\nabla_\sigma h,\nabla_\sigma\psi\rangle}{D}du\nonumber\\
&+&\int_\sfe Ahc_{ij,kl}(h)(\psi_{ij}+\psi\delta_{ij})(\psi_{kl}+\psi\delta_{kl})du\nonumber\\
&+&\int_\sfe h\det(Q)C\left[\frac{h\psi+\langle\nabla_\sigma h,\nabla_\sigma\psi\rangle}{D}\right]^2du\nonumber\\
&+&\int_\sfe h\det(Q)B\left[D(h^2+|\nabla_\sigma\psi|^2)-
\frac{[h\psi+\langle\nabla_\sigma h,\nabla_\sigma\psi\rangle]^2}{D}
\right]\frac1{D^2}\,du.
\end{eqnarray}

We now focus on the fourth summand of the last expression. Applying formulas \eqref{ibp2} and \eqref{homog2} we get
\begin{eqnarray*}
&&\int_\sfe Ahc_{ij,kl}(h)(\psi_{ij}+\psi\delta_{ij})(\psi_{kl}+\psi\delta_{kl})du\\
&=&\int_{\sfe}\psi c_{ij,kl}(h)(\psi_{ij}+\psi\delta_{ij})((Ah)_{kl}+Ah\delta_{kl})du\\
&=&\int_\sfe\psi c_{ij,kl}(h)(\psi_{ij}+\psi\delta_{ij})(A(h_{kl}+h\delta_{kl})+2A_kh_l+h A_{kl})du\\
&=&\int_\sfe A\psi c_{ij,kl}(h)(\psi_{ij}+\psi\delta_{ij})(h_{kl}+h\delta_{kl})du \nonumber\\ &+&
\int_\sfe\psi c_{ij,kl}(h)(\psi_{ij}+\psi\delta_{ij})(2A_k h_l+hA_{kl})du\\
&=&(n-2)\int_\sfe A\psi c_{ij}(h)(\psi_{ij}+\psi\delta_{ij})du \nonumber\\ &+&
\int_\sfe\psi c_{ij,kl}(h)(\psi_{ij}+\psi\delta_{ij})(2A_k h_l+hA_{kl})du.
\end{eqnarray*}
Hence
\begin{eqnarray}
g''(0)&=&n\int_\sfe\psi c_{ij}(h)(\psi_{ij}+\psi\delta_{ij})Adu+ 2\int_\sfe\psi\det(Q)B
\frac{h\psi+\langle\nabla_\sigma h,\nabla_\sigma\psi\rangle}{D}du\nonumber\\
&+&2\int_\sfe h c_{ij}(h)(\psi_{ij}+\psi\delta_{ij})B\frac{h\psi+\langle\nabla_\sigma h,\nabla_\sigma\psi\rangle}{D}du\nonumber\\
&+&\int_\sfe \psi c_{ij,kl}(h)(\psi_{ij}+\psi\delta_{ij})(2A_k h_l+hA_{kl})du \nonumber\\
&+&\int_\sfe h\det(Q)C\left[\frac{h\psi+\langle\nabla_\sigma h,\nabla_\sigma\psi\rangle}{D}\right]^2du\nonumber\\
&+&\int_\sfe h\det(Q)B\left[D(\psi^2+|\nabla_\sigma\psi|^2)-
\frac{[h\psi+\langle\nabla_\sigma h,\nabla_\sigma\psi\rangle]^2}{D}
\right]\frac1{D^2}\,du.
\end{eqnarray}

Let $h\equiv R$, $R>0$. This choice considerably simplifies the situation as:
\begin{eqnarray*}
&&Q=R I_{n-1};\quad\nabla_\sigma\equiv R;\quad D\equiv R;\quad c_{ij}(h)\equiv R^{n-2}\delta_{ij};\\
&&A=\int_0^1t^{n-1}f(Rt)dt,\quad
B=\int_0^1t^nf'(Rt)dt,\quad
C=\int_0^1t^{n+1}f''(Rt)dt.
\end{eqnarray*}
Here $I_{n-1}$ denotes the $(n-1)\times(n-1)$ identity matrix.
In particular $A$ does not depend on the point $u$ on $\sfe$, so that
$$
A_i\equiv A_{ij}\equiv0\quad\mbox{on $\sfe$.}
$$
Hence $g(0)=|\sfe|R^{n}A$, and
\begin{eqnarray}\label{first derivative at zero}
g'(0) & = & R^{n-1}A \int_\sfe \psi du + R^{n-1}A \int_\sfe (\Delta_\sigma\psi + (n-1)\psi) du + R^n B \int_\sfe \psi du \nonumber\\ & = & R^{n-1}(nA+RB)\int_\sfe \psi du.
\end{eqnarray}
Here we 
used the fact that, by the divergence theorem on $\sfe$,
$$
\int_\sfe\Delta_\sigma\psi du=0.
$$
As for the second derivative, we have
\begin{eqnarray*}
g''(0)&=&nR^{n-2}A\int_\sfe\psi(\Delta_\sigma\psi+(n-1)\psi)du+2R^{n-1}B\int_\sfe\psi^2du\nonumber\\
&+&2R^{n-1}B\int_\sfe\psi(\Delta_\sigma\psi+(n-1)\psi))du+R^nC\int_\sfe\psi^2du\nonumber\\
&+&R^{n-1}B\int_\sfe|\nabla_\sigma\psi|^2du.
\end{eqnarray*}
By the divergence theorem, 
$$
\int_\sfe\psi\Delta_\sigma\psi du=-\int_\sfe|\nabla_\sigma\psi|^2du,
$$
and thus
\begin{equation}\label{second derivative at zero}
g''(0)=R^{n-2}(An(n-1)+2nRB+R^2C)\int_\sfe\psi^2du - R^{n-2}(nA+RB)\int_\sfe|\nabla_\sigma\psi|^2du.
\end{equation}
Integrating by parts in $t$, we get
$$
f(R)=nA+RB,
$$
and
$$
f'(R)=(n+1)B+RC.
$$
Thus we obtain
\begin{equation}\label{first derivative at zero2}
g'(0)=R^{n-1}f(R)\int_\sfe\psi du,
\end{equation}
and
\begin{eqnarray}\label{second derivative at zero2}
g''(0) & = & R^{n-2}\left[(n-1)f(R)+R f'(R)\right]\int_\sfe\psi^2du - R^{n-2}f(R)\int_\sfe|\nabla_\sigma\psi|^2du \nonumber\\ & = & R^{n-2}f(R) \left( (n-1)\int_\sfe\psi^2du - \int_\sfe|\nabla_\sigma\psi|^2du \right) + R^{n-1} f'(R) \int_\sfe\psi^2du. 
\end{eqnarray}
This concludes the proof of Lemma \ref{L-T-1}.

\subsection{Proof of the Lemma \ref{L-T-2}}

Firstly, we state the following. 

\begin{lemma}\label{BM->Log-BM}
Let $n\geq 2.$ Let $\gamma$ be a measure on $\R^n.$ Fix $h\in C^{2,+}(\sfe)$, $\varphi\in C^2(\sfe)$, $\varphi>0$ and 
set $\psi=h\log\varphi$. Let $\textbf{K}(h,\psi,I)$, with $I=[-a,a]$ and $a>0$, be the corresponding 
one-parameter family. Consider the function $f(s)=\gamma(K_s)$. 
Introduce the additional notation for the operator $F(h,\psi):=f'(0).$ Set 
\begin{equation}\label{refnow}
A(h,\psi):=\left.\frac{d F\left(h,\frac{h+s\psi}{h}\psi\right)}{ds}\right|_{s=0}.
\end{equation}

Consider the one-parameter family $\textbf{Q}(h,\varphi,[-a,a])$, i.e. the collection of sets with support functions 
$h_s=h\varphi^s$, $s\in[-a,a]$. Let $g(s)=\gamma(Q_s)$. Then
\begin{itemize}
\item $g(0)=f(0)$;
\item $g'(0)=f'(0);$
\item $g''(0)=f''(0)+A(h,\psi).$
\end{itemize} 
\end{lemma}

The proof of the Lemma immediately follows from the fact that
$$
h\varphi^s=h+sh\log\varphi+o(s),\quad\mbox{as $s\to0$,}
$$
with the selection $\psi=h\log \varphi$. When $h\equiv R>0$, the additional term introduced in Lemma \ref{BM->Log-BM} can be written as follows:
$$
A(h,\psi)=f(R)\int_{\sfe} \psi^2 du.
$$
That, together with Lemma \ref{L-T-1}, implies Lemma \ref{L-T-2}.

\subsection{Additional remarks.}

Finally, we note that Lemma \ref{BM->Log-BM} implies the following result.

\begin{theorem}[Infinitesimal form of Log-Brunn-Minkowski conjecture]\label{Log_BM_infinitesimal_main_TH}
Let $n\geq 2$ be an integer. If Conjecture \ref{Log-BM} is true, then for every $h\in C^{2,+}_e(\sfe)$, $\psi\in C^{2}(\sfe)$, $\psi$ even and strictly
positive,
\begin{equation}\label{Log_BM_inf_eq}
\int_{\sfe} \psi^2\frac{1+\trace(Q^{-1}(h)) h}{h^2}d\bar{V}_h-n\left(\int_{\sfe} \frac{\psi}{h}d\bar{V}_h\right)^2\leq \int_{\sfe} \frac{1}{h}\langle Q^{-1}(h)
\nabla \psi, \nabla \psi\rangle d\bar{V}_h.
\end{equation}
Here $h$ is the support function of $K$, $Q(h)$ is the curvature matrix of $K$ and 
$$d\bar{V}_h=\frac{1}{|K|}\frac{1}{n} h_K(u)\det Q(h_K(u))du$$ 
is the normalized cone measure of the convex body $K$.
\end{theorem}

A corresponding infinitesimal Brunn-Minkowski inequality for Lebesgue measure was obtained by the first named author in \cite{Col1} and reads as:
\begin{equation}\label{BM_inf_eq}
\int_{\sfe} \psi^2\frac{\trace(Q^{-1}(h))}{h}d\bar{V}_h-(n-1)\left(\int_{\sfe} \frac{\psi}{h}d\bar{V}_h\right)^2\leq 
\int_{\sfe} \frac{1}{h}\langle Q^{-1}(h)\nabla \psi, \nabla \psi\rangle d\bar{V}_h.
\end{equation}
Note that by the Cauchy-Schwarz inequality, 
$$\int_{\sfe} \frac{\psi^2}{h^2}d\bar{V}_h\geq \left(\int_{\sfe} \frac{\psi}{h}d\bar{V}_h\right)^2.$$
Hence, (\ref{Log_BM_inf_eq}) is indeed a strengthening of (\ref{BM_inf_eq}). 

In particular, letting $\varphi\equiv1$ we arrive to the following corollary of Theorem \ref{Log_BM_infinitesimal_main_TH}.

\begin{corollary}[A strengthening of Minkowski's second inequality.]
Let $K$ be a convex symmetric set in the plane, or a convex unconditional set in $\R^n$. Then,
\begin{equation}\label{AF_uncond}
V_n(K)\left(V_{n-2}(K)+\int_{\partial K} \frac{1}{\langle y, \nu_K(y)\rangle}d\sigma(y)\right)\leq V_{n-1}(K)^2,
\end{equation}
where $V_{n-i}$ are the intrinsic volumes of $K$, $\nu_K(y)$ stands for the unit normal at $y\in\partial K$ and $d\sigma(y)$ is the surface area measure on $\partial K$.
\end{corollary}

Minkowski's second inequality, which states that for every convex set $K \subset \R^n$ one has
$$ 
V_n(K) V_{n-2}(K) \leq \frac{n-1}{n} V_{n-1}(K)^2, 
$$
is deduced from (\ref{AF_uncond}) by using the Cauchy-Schwarz inequality.
For a more general version of this inequality see, for example, Schneider \cite[Chapter 4]{book4}.

\end{document}